\documentclass[12pt,a4paper]{amsart}
 \usepackage{amssymb,amsthm}
 \usepackage{multirow}
 \usepackage{dsfont}
 \usepackage{graphicx}
 \usepackage{float}
 \newtheorem{thm}{Theorem}[section]
\newtheorem{cor}[thm]{Corollary}
\newtheorem{prop}[thm]{Proposition}
\newtheorem{fact}[thm]{Fact}
\newtheorem{lem}[thm]{Lemma}

\newtheorem{definition}{Definition}

\theoremstyle{definition}

\begin{document}
 \title
 {$k$-Color Region Select Game}
 \author{Ahmet Batal, Neslihan G\"ug\"umc\"u}
 \address{Department of Mathematics\\ Izmir Institute of Technology\\ G\"ulbah\c ce Campus 35430 Izmir, TURKEY}
\email{ahmetbatal@iyte.edu.tr}
\email{neslihangugumcu@iyte.edu.tr}
\begin{abstract}
The region select game, introduced by Ayaka Shimizu,  Akio Kawauchi and Kengo Kishimoto, is a game that is played on knot diagrams whose crossings are endowed with two colors. The game is based on the region crossing change moves that induce an unknotting operation on knot diagrams. We generalize the region select game to be played on a knot diagram endowed with $k$-colors at its vertices for $2 \leq k \leq \infty$.

\end{abstract}

 \maketitle

 \section*{Introduction}

 The \textit{region select game} that was produced in 2010 \cite{Shi2, Shi} and later released as a game app for Android \cite{And}, is a game played on knot diagrams. The region select game begins with a knot diagram that is initially endowed with two colors, either by $0$ or $1$, at its crossings, and played by selecting a number of regions (an area enclosed by the arcs of the diagram) of the knot diagram. Each choice of a region  of the diagram results in the crossing colors which lie on the boundary of the region to be increased by $1$ modulo $2$. The aim of the game is to turn the color of each crossing of the knot diagram to $0$ (or to $1$) by selecting a number of regions. Shimizu showed \cite{Shi} that the region select game is always solvable, that is, for any initial color configuration of crossings there exists a choice of regions which turns the color of each crossing to $0$. In \cite{Shi} a \textit{region crossing change} move is defined to be a local transformation of the knot diagram that is applied on a region and  changes the type of each crossing that lie on the boundary of the region. By encoding an over-crossing with $1$ and an under-crossing with $0$, it is clear that any knot diagram corresponds to a knot diagram given with an initial color configuration at its crossings. The solvability of the region select game follows from the result of Shimuzi that any knot diagram can be turned into an unknot diagram by a sequence of region crossing change moves \cite{Shi}. In \cite{Che}, Cheng and Gao showed that the result holds for two-component link diagrams if and only if their linking numbers are even.

Soon after in 2012, Ahara and Suzuki \cite{AhSu} extended the region select game to an integral setting by introducing the \textit{integral region choice problem}. In the integral choice problem,  one starts the game with a knot diagram that is endowed with colors  labeled by integers at its crossings. Then, an integer is assigned to a region of the knot diagram. The assigned integer on the region changes the integer label on the crossings that lie in the boundary of the region according to two counting rules. In the first counting rule, named as \textit{the single counting rule}, the integer label on each crossing of the boundary of the integer-labeled region is increased by $n$, where $n$ is the integer assigned to the region. In the second counting rule, named as \textit{the double counting rule}, when an integer is assigned to a region, the integer labels on the crossings of the boundary that meet with the region once are increased by $n$, and the integer labels on the crossings of the boundary that meet with the region twice are increased by $2n$. In \cite{AhSu}, the authors showed that the integral region choice problem considered with respect to both of these rules is always solvable. In \cite{Kaw}, Kawamura gave a necessary and sufficient condition for the solvability of two-component links diagrams.

In this paper, we introduce the $k$-color region select game that is the modulo $k$ extension of Shimizu's region select game, when $k$ is an integer greater than $2$.  In this game, crossings of a knot diagram are initially colored by integers $0,1,...,k-1$. The game is played by pushing (selecting) a number of regions of the knot diagram. Each push of a region increases the color of the crossings at the boundary of the region by $1$ modulo $k$. The aim of the game is to make the color of every crossing $0$ by applying a push pattern to the regions. See Figure \ref{fig:example} for a knot diagram given with an initial coloring configuration. The integers on the regions of the knot diagram denote the required number of pushes on them to turn each vertex color to $0$ modulo $3$. Similar to the integral region choice problem of Ahara and Suzuki, we also define versions of the game for the cases $2 \leq k< \infty$ and $k=\infty$ with modified rules of counting.

\begin{figure}[H]
\centering
\includegraphics[scale=.2]{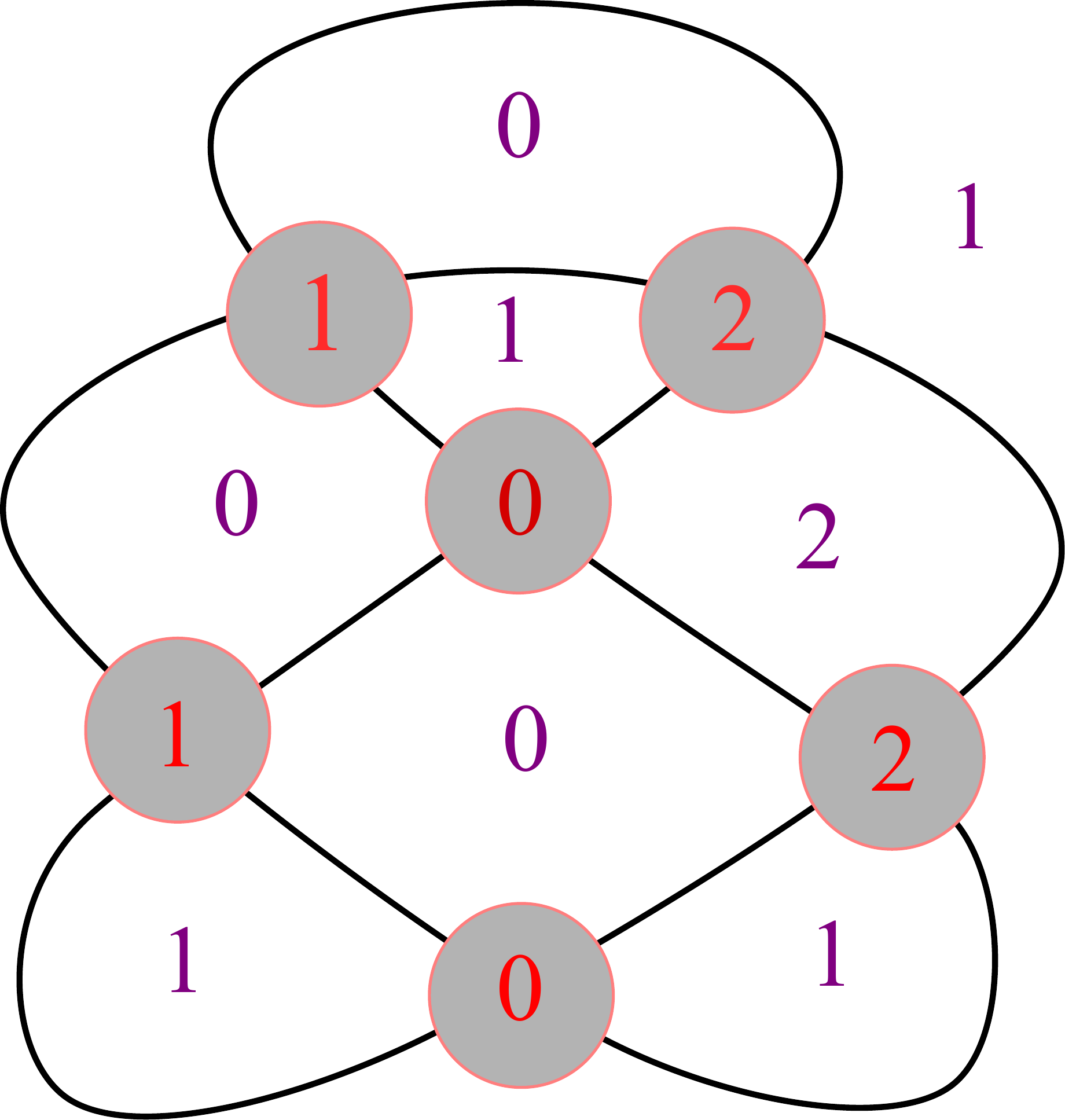}
\caption{A solution of a 3-color region select game played on a diagram of the knot $6_2$ \cite{Rotable}}
\label{fig:example}
\end{figure}

Let us now give an outline of our paper.

In Section \ref{sec:prem}, we review fundamental notions from knot theory and graph theory that are required throughout the paper.  In Section \ref{sec:game}, we introduce the $k$-color region select game both for an integer $k$ that is greater than or equal to $2$ and for $k=\infty$.

In Section \ref{sec:solvable} we prove that any version of the $k$-color region select game introduced in this paper, is always solvable on knot diagrams. In Sections \ref{sec:reduced} and \ref{sec:assertions} we examine the number of solving patterns for a given initial coloring configuration that are obtained without pushing certain regions of a knot diagram.

We note here that the always solvability of the $k$-color region select game with the versions corresponding to the single and double counting rule, can be directly deduced from the always solvability of the integral region choice problem.  However this does not make our proof redundant.
 In fact,  the proofs of the always solvability of the integral region choice problem and the original ($2$-color) region select game are mostly knot theoretic where Reidemeister moves and checkerboard shadings of knot diagrams are used. On the other hand, our proof utilizes mostly linear algebra and few fundamental facts on regular curves (indeed we almost only utilize the fact that a knot diagram is an orientable closed curve). This enables us to prove the always solvability of the other versions of the region select game that are introduced in this paper that cannot be drawn directly from the arguments in \cite{AhSu}. In particular, with our proof method we also prove the always solvability of the integral region choice problem, not only for the single and double counting rule, but also for any arbitrary counting rule. With the arguments in our paper, the following questions are also answered.

\begin{enumerate}
\item How many solutions are there for a given initial color configuration?

\item  How many initial color configuration can we solve without pushing certain regions?

\item  Do there exist certain regions such that any initial color configuration can be solved without pushing them?

\item Do the answers of the above questions depend on the value of $k$, the version of the game, and the type of the knot diagram? If so, how?

\end{enumerate}

\section{Preliminaries}\label{sec:prem}

We shall begin by presenting basic definitions that we will be using throughout the paper.

\begin{definition}\normalfont
A \textit{link} with $n$ components is a smooth embedding of a union of $n$ unit circles, $S^1$ into $\mathbb{R}^3$, where $n \geq 1$.  In particular, a link with one component is called a \textit{knot}.

\end{definition}

\begin{definition}\normalfont
A \textit{link diagram} (or a \textit{knot diagram}) $D$ is a regular projection of a link (or a knot) into the $2$-sphere, $S^2$ with a finite number of transversal self-intersection points. Each self-intersection point of the projection curve is endowed either with over or under passage information to represent the weaving of the link in $\mathbb{R}^3$, and is called a \textit{crossing} of $D$.

\end{definition}




\begin{definition}\normalfont
A crossing of a link diagram is called \textit{reducible} if there exists a circle in the plane of the diagram that meets the diagram transversely only at that crossing. A crossing is called \textit{irreducible} if it is not reducible.

\end{definition}

\begin{definition} \normalfont
We call a component of a \textit{link diagram} without any crossing on it a \textit{loop}. 
\end{definition}
It is clear that a loopless link diagram with $n$ crossings overlies a unique planar graph with $n$ four-valent vertices that is obtained by ignoring the weaving information at the crossings. By abusing the terminology,  we extend this association to any link diagram by considering each loop component as a graph with one edge and no vertices. We also call the underlying graph of a link or a knot diagram a \textit{link diagram} or a \textit{knot diagram}, respectively. By a simple application of the Jordan curve theorem and Euler's formula, one can see that any knot diagram with $n$ vertices divides $S^2$ into $n+2$ regions for $n \geq 0$.

\begin{definition}
\normalfont
 For a link diagram $D$ on $S^2$, \textit{regions} of $D$ are defined as the connected components of $S^2 \backslash D$. A vertex $v$ (an edge $e$) is said to be \textit{incident} to a region $r$ and vice versa if $v$ ($e$, respectively) is in the boundary of $r$. Two regions of $D$ are called \textit{adjacent} if they are incident to the same edge. Similarly,  two edges of $D$ are called  \textit{adjacent} if they are incident to the same vertex.
\end{definition}

\begin{definition}\normalfont
Let $D$ be a link diagram.
The \textit{dual graph} of $D$ is the graph obtained by adding a vertex to each region of $D$ and an edge between each pair of vertices that lie on adjacent regions.
\end{definition}

\begin{figure}[H]
\centering
\includegraphics[scale=.2]{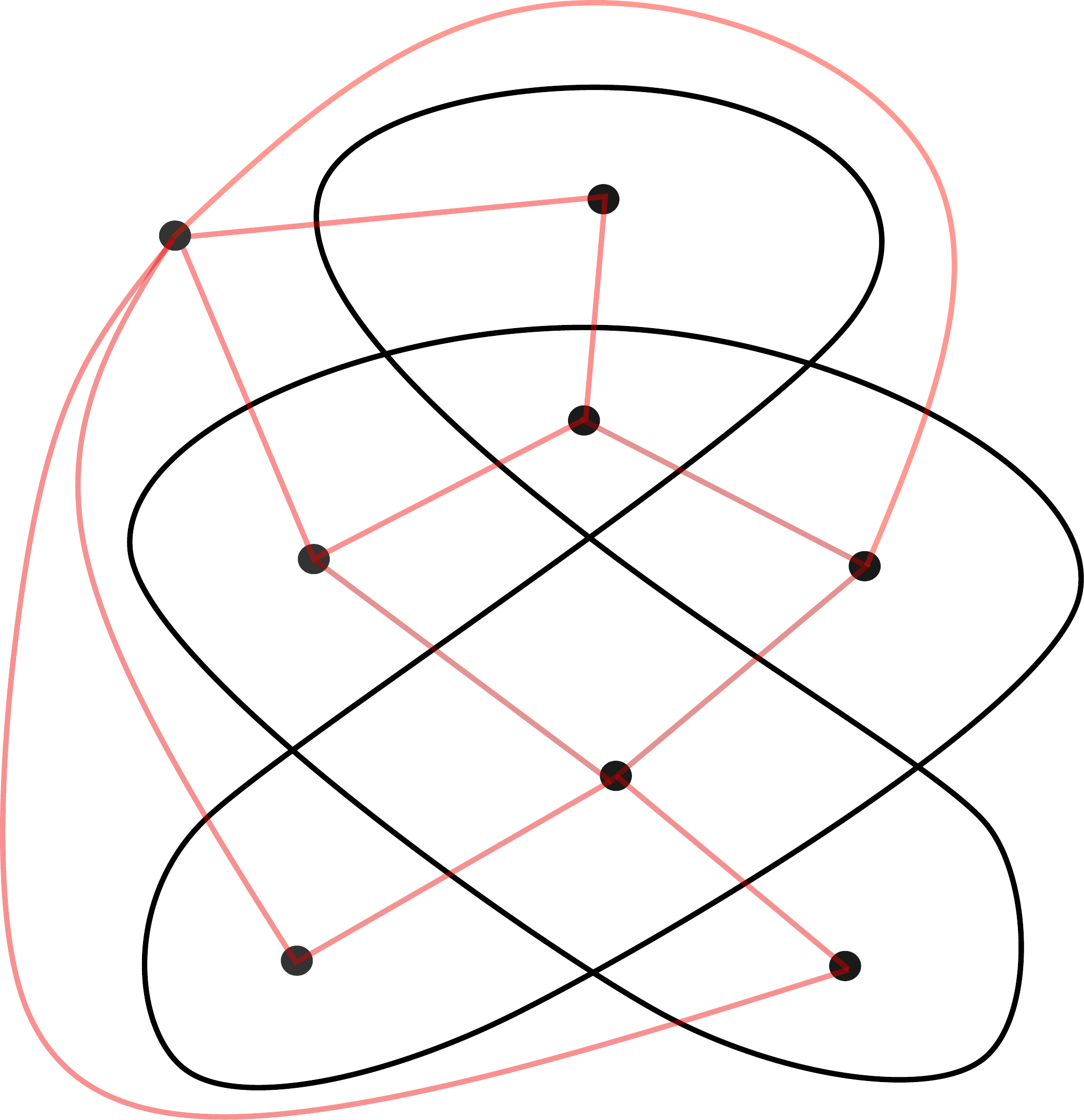}
\caption{The dual graph of a diagram of the knot $6_2$}
\label{fig:example}
\end{figure}

 \section{$k$-color region select game}\label{sec:game}

 In this section, we introduce \textit{$k$-color region select game} as well as its modified versions that are all played on a knot diagram $D$, both for the cases  $2 \leq  k < \infty$ and $k=\infty$.\\

\textit{The $k$-color region select game when $2 \leq  k < \infty$} :\\
We select $k$ colors and give a labeling to these colors as $color\,0,\, color \,1, ..., \\
color\, k-1$. Then we take an initial color configuration of vertices of $D$ by using these colors. The game is played by pushing regions of $D$. When a region is pushed, every vertex incident to the region changes its color by the following rule. The $color \,i$ changes to the $color \,i+1$ for $i\neq k-1$ and the $color\, k-1$ changes to the $color\, 0$. The aim of the game is to reach to the \textit{off color} configuration, in which every vertex is in $color \,0$ state, by applying a push pattern on regions for a given initial color configuration.
\\

\textit{The $k$-color region select game when $k=\infty$}:\\
In this game, we have infinitely many colors labeled as
$...,color\,-2,\, color \,-1,\, color \,0,\, color \,1,\,color \,2,...$. An initial color configuration of vertices of $D$ is obtained by a finite choice of these colors. Each push of a region is assigned to either to $1$ or $-1$, and is called a \textit{positive} or \textit{negative} push, respectively.

When a positive (negative, respectively) push is applied to a region, color label of every vertex incident to  the region increases (decreases, respectively) by $1$. The aim of the game is the same as in the finite case, to reach to the off color configuration by applying a signed push pattern for a given initial color configuration.
\begin{definition}\normalfont
Let $C$ denote an initial color configuration of a link diagram $D$.
If there exists a push pattern $P$ of regions of $D$ which brings $C$ to the off color configuration then $C$ is  called \textit{solvable} and $P$ is a solving pattern for $C$.
\end{definition}
\begin{definition}\normalfont
 If every initial color configuration of vertices of $D$ is solvable then $D$ is called \textit{always solvable} in the $k$-color region select game.
 \end{definition}

Let $D$ have $n$ vertices and $m$ regions and let us enumerate the vertices and regions of $D$ as  $\{v_1,...,v_n\}$, $\{r_1,...,r_m\}$, respectively.
It is easy to observe that the order of the pushes has no importance. Moreover, for $k<\infty$, pushing a region $k$ times is equivalent to not to push it. For $k=\infty$, the net  number of  pushes, that is equal to the sum of signs of the pushes made, 
 is important. Precisely, the color label of the vertices that are incident to the regions pushed change by the net number of pushes.

Let $\mathbb{Z}_k$ denote the quotient ring ${\mathbb{Z}} /{ k \mathbb{Z}}$ when $k<\infty$, and it denotes $\mathbb{Z}$ when  $k=\infty$.

 We identify a push pattern of regions by a column vector $\mathbf{p}=(p_1,..., p_m)^t \in \mathbb{Z}_k^m$ such that $\mathbf{p}(r_i):=p_i$ is the number of times the region $r_i$ is pushed modulo $k$ if $k<\infty$, and the net  number of pushes of $r_i$ if $k=\infty$. Similarly, we identify a color configuration of vertices by a column vector  $\mathbf{c}=(c_1,..., c_n)^t \in \mathbb{Z}_k^n$ such that $\mathbf{c}(v_i)=c_i$ is the label number of the color of the vertex $v_i$ in the configuration.

The $n \times m$ \textit{vertex-region incidence matrix} $M_0=M_0(D)$ of $D$ is constructed as follows \cite{Che}

\begin{align}
(M_0)_{ij}= \left\{ \begin{array}{cc}
                1 &  \;\;\text{if}\;\;\;\; v_i\; \text{is incident to}\; r_j     \\
                0 &  \;\;\text{otherwise}  \\
                \end{array} \right\}.
\end{align}

Let $\mathbf{c}_{in}$ be an initial color configuration of vertices of $D$ and $\mathbf{c}_{fin}$ be the final state configuration obtained after applying a push pattern $\mathbf{p}$. One can observe that the following equation characterizes the relation among $\mathbf{c}_{in}$,$\mathbf{c}_{fin}$, and $\mathbf{p}$ over $\mathbb{Z}_k$  in a simple algebraic way.
\begin{equation}
\label{maineqn}
\mathbf{c}_{in}+M_0(D)\mathbf{p}=\mathbf{c}_{fin}.
\end{equation}

We now introduce the modified versions of the game that are played with the rules explained below. 
\\

 \emph{Modified rules of the game for $k<\infty$}: 
 
 Take a link diagram $D$ and fix some $k<\infty$. Let $v$ be a vertex of $D$.
 \begin{enumerate}
 \item If $v$ is irreducible choose a number $a\in \mathbb{Z}_k$ which is not a zero divisor. Then define the new rule for this vertex such that a push on a region incident to $v$ increases the color label of $v$ by $a$ modulo $k$.
 \item If $v$ is reducible, choose three numbers $a_0, a_1, a_2 \in \mathbb{Z}_k$ such that $a_1$ and $a_2$ are not zero divisors. Let $r_0$, $r_1$, and $r_2$ be the regions incident to $v$ where $r_0$ is the region which touches $v$ from two sides. Then define the rule for this vertex such that a push on the incident region $r_i$ increases the color label of $v$ by $a_i$ modulo $k$ for $i=0,1, 2$. Let us call these numbers we choose for each vertex region pair $v$-$r$, \emph{the increment number of} $v$\emph{ with respect to the region} $r$ or \emph{the increment number of} $v$-$r$ \emph{pair}. Note that the increment number of $v$ is the same with respect to each incident region of $v$ if $v$ is irreducible, but it can be chosen differently for each incident region of $v$ if $v$ is reducible.
     \\
     
 \end{enumerate}

\emph{Modified rules of the game for $k=\infty$}:

\begin{enumerate}
\item

 The increment number of the incident vertex-region pairs $v$-$r$ is taken $1$ as in the original game if $v$ is irreducible, or if $v$ is reducible and $r$ is a region which touches $v$ from one side.

 \item If $v$ is a reducible vertex and $r$ is the region which touches $v$ from two sides, then the increment number of $v$-$r$ pair is allowed to be any number.
 \end{enumerate}

The rules mentioned above and every choice of increment numbers induce different versions of the $k$-color region select game for $2 \leq k\leq\infty$. The game where all increment numbers are taken as $1$ corresponds to the original game, hence the modified versions are generalizations of the original game for $2 \leq k\leq\infty$. Although the complexity of the game is increased by these modifications, it will turn out that always solvability of the game is not affected as we show in Section \ref{sec:solvable}. Therefore, in the sequel, we consider the modified versions of the game.


Note also that in the case of $k=\infty$, we allow the increment number of $v$-$r$ pair where $v$ is a reducible vertex and $r$ is the region which touches $v$ from two sides to be any number. When this number taken as $1$ or $2$ for all reducible vertices, the version corresponds to the integral choice problem for single or double counting rule \cite{AhSu}, respectively.

\begin{definition}\normalfont
 Let $D$ be a link diagram with vertices labeled as $\{v_1,...,v_n\}$ and regions $\{r_1,...,r_{m}\}$  and $G$ be a version of the $k$-color region select game  on $D$ induced by the choice of $k$ and the set of increment numbers. We define the  \textit{game matrix} $M=M(D,G)$ \emph{over} $\mathbb{Z}_k$ \emph{corresponding to the  diagram} $D$ and \emph{the game} $G$ such that $(M)_{ij}$ is equal to the increment number of the vertex $v_i$ with respect to the region $r_j$ if $v_i$ and $r_j$ are incident, and zero otherwise.
 \end{definition}

Similar to the original game, in the game $G$, a final state color configuration $\mathbf{c}_{fin}$ is obtained after applying a push pattern ${\bf p}$ to an initial color configuration $\mathbf{c}_{in}$ if and only if

\begin{equation}
\label{maineqn2}
\mathbf{c}_{in}+M(D,G)\mathbf{p}=\mathbf{c}_{fin} \;\;\text{over} \;\; \mathbb{Z}_k.
\end{equation}

Let us denote the kernel and column space of a matrix $A$ over the ring $\mathbb{Z}_k$ by $Ker_k(A)$ and $Col_k(A)$, respectively. Then, from the above algebraic formulation we immediately obtain the following facts.
\begin{fact}
An initial color configuration $\mathbf{c}$ of the vertices of $D$ is solvable in the game $G$ if and only if $\mathbf{c}\in Col_k(M)$. Indeed, $\mathbf{p}$ is a solving pattern for $\mathbf{c}$ if and only if
\begin{equation}
M\mathbf{p}=-\mathbf{c}.
\end{equation}
\end{fact}
\begin{fact}
\label{fact2}
$D$ is always solvable in $G$ if and only if $Col_k(M)=\mathbb{Z}_k^n$.
\end{fact}
\begin{fact}
\label{fact3}
In the case $k<\infty$, for every solvable configuration $\mathbf{c}$, there exist exactly $s$ solving patterns where $s= |Ker_k(M)|$.
\end{fact}

We also have the following proposition.
\begin{prop}
\label{propker}

In the case $k<\infty$, $D$ is always solvable in $G$ if and only if $|Ker_k(M)|=k^{m-n}$.

\end{prop}

\begin{proof}
Since the matrix multiplication is a homomorphism of modules, by the fundamental theorem of homomorphisms we have\\
 $$|Col_k(M)||Ker_k(M)|=|\mathbb{Z}_k^m|=k^m.$$
Then the result follows by Fact \ref{fact2}.
\end{proof}

\begin{definition}\normalfont
Let $A$ be a matrix over $\mathbb{Z}_k$, where $k\leq \infty$. A pattern is called a \emph{null pattern} of $A$ if it belongs to $Ker_k(A)$.
\end{definition}

We have the following proposition.

\begin{prop}
\label{propmn}
Let $D$ be a link diagram with $n$ vertices and $m$ regions on which we play a version of the $k$-color region select game $G$ where $k< \infty $. Let $M$ be the corresponding game matrix. Fix $i \geq 0$ regions of $D$. Let $j$ be the number of null patterns of $M$ where these regions are not pushed. Then, there are $k^{m-i}/ j$ initial color configurations that can be solved without pushing these regions.

If there are $m-n$ regions where the only null pattern of $M$ these regions are not pushed is the trivial pattern $\mathbf{0}$, then, $D$ is always solvable in $G$. Moreover, any initial color configuration can be solved uniquely without pushing these regions.

\end{prop}
\begin{proof}

 Take an enumeration of the regions of $D$ such that the regions we fix are $r_{m-i+1},..., r_m$. For a vector $\mathbf{p}=(p_1,...,p_{m-i})^t\in\mathbb{Z}^{m-i}_k$, define the zero extension vector $\mathbf{p_e}=(p_1,...,p_{m-i},0,...,0)^t\in\mathbb{Z}^m_k$. Let $\widetilde{M}$ be the $n\times (m-i)$ matrix obtained from $M$ by deleting the last $i$ columns. Then, $\widetilde{M}\mathbf{p}=M\mathbf{p_e}$. Therefore $\mathbf{p}\in Ker_k(\widetilde{M})$ if and only if $\mathbf{p_e}\in Ker_k(M)$.
Hence, $j=|Ker_k(\widetilde{M})|$. Moreover, if an initial color configuration can be solved without pushing the regions $r_{m-i+1},..., r_m$, it must belong to $Col_k(\widetilde{M})$. On the other hand, $|Col_k(\widetilde{M})|= k^{m-i} / |Ker_k(\widetilde{M})|$ by the fundamental theorem of homomorphisms. Hence, there are $k^{m-i}/ j$ number of initial color configurations that can be solved without pushing these regions.

 If there are $m-n$ regions where the only null pattern these regions are not pushed is the trivial pattern $\mathbf{0}$, then $i=m-n$ and $j=1$. Hence,  $Ker_k(\widetilde{M})=\{\mathbf{0}\}$, and $ |Col_k(\widetilde{M})|= k^n $. Since $k^n$ is the number of all possible initial color configurations, this implies that any initial color configuration can be solved uniquely without pushing these regions. In particular, $D$ is always solvable.

\end{proof}

\section{Knot Diagrams are always solvable}\label{sec:solvable}

In this section, we show that knot diagrams are always solvable with respect to any version of the $k$-color region select game for any $k \leq \infty$.

\begin{definition}\normalfont
For a fixed $k\leq \infty$,  a vertex $v$ is said to be \emph{balanced} with respect to a push pattern $\mathbf{p}$ if the sum of the pushes of regions incident to $v$ is zero modulo $k$ in $\mathbf{p}$.
\end{definition}

\begin{lem}
\label{lem:bal}
Let $M$ be a game matrix of a link diagram $D$ over $\mathbb{Z}_k$, where $k\leq\infty$, and $\boldsymbol{\ell}$ be a null pattern of $M$. Then, any irreducible vertex of $D$ is balanced with respect to $\boldsymbol{\ell}$.
\end{lem}
\begin{proof}
Let $v$ be an irreducible vertex of $D$ and let $a$ be the increment number of $v$ with respect to all its incident regions in the version of the $k$-color region select game corresponding to $M$. Let $r_1,...,r_4$ be the regions incident to $v$. Then, $(M\boldsymbol{\ell})(v)= a(\boldsymbol{\ell}(r_1)+\boldsymbol{\ell}(r_2)+\boldsymbol{\ell}(r_3)+\boldsymbol{\ell}(r_4))$. On the other hand, since $\boldsymbol{\ell}$ is a null pattern of $M$, $M\boldsymbol{\ell}=0$. Hence $a(\boldsymbol{\ell}(r_1)+\boldsymbol{\ell}(r_2)+\boldsymbol{\ell}(r_3)+\boldsymbol{\ell}(r_4))=0$. By the rules of the game $a=1$ if $k=\infty$ and $a$ is not a zero divisor of $\mathbb{Z}_k$ for $k<\infty$. Hence, $\boldsymbol{\ell}(r_1)+\boldsymbol{\ell}(r_2)+\boldsymbol{\ell}(r_3)+\boldsymbol{\ell}(r_4)=0$, which means $v$ is balanced.
\end{proof}

\begin{definition}\normalfont
The \emph{push number} $\sigma_{\bf p}(e)$ \emph{of an edge} $e$ \emph{with respect to a push pattern} ${\bf p}$ is the sum of the pushes of the regions incident to $e$ in ${\bf p}$ modulo $k$. More precisely, if $e$ is incident to the regions $r_1$ and $r_2$, then $\sigma_{\bf p}(e)= {\bf p}(r_1)+ {\bf p}(r_2)$ $\mod$ $k$. \end{definition}
We have the following lemma.

\begin{lem}
\label{lempush}
Let $D$ be an oriented reduced knot diagram and $\boldsymbol{\ell}$ be a null pattern of a game matrix $M$ of $D$ over $\mathbb{Z}_k$, where $k\leq \infty$. Then, there exists $s\in \mathbb{Z}_k$  such that $\sigma_{\boldsymbol{\ell}}(e)=s$ or $-s$ for every edge $e$ of $D$. Moreover, for any pair of adjacent edges $e_1$ and $e_2$ which are not incident to the same region, $\sigma_{\boldsymbol{\ell}}(e_1)=s$ if and only if  $\sigma_{\boldsymbol{\ell}}(e_2)=-s$.
\end{lem}

\begin{proof}
Let $e_1$ and $e_2$ be two adjacent edges that meet at a vertex $v$ and are not incident to the same region. Let $r_1,...,r_4$ be the regions incident to $v$ such that $r_1$ and $r_2$ are incident to $e_1$, $r_3$ and $r_4$ are incident to $e_2$. Let $\sigma_{\boldsymbol{\ell}}(e_1)=s$, for some $s\in\mathbb{Z}_k $. This means $\boldsymbol{\ell}(r_1)+\boldsymbol{\ell}(r_2)=s$. On the other hand, since $D$ is a reduced knot diagram, $v$ is an irreducible vertex. Hence by Lemma \ref{lem:bal}, it is balanced with respect to $\boldsymbol{\ell}$, i.e; $\boldsymbol{\ell}(r_1)+\boldsymbol{\ell}(r_2)+\boldsymbol{\ell}(r_3)+\boldsymbol{\ell}(r_4)=0$. This implies $\sigma_{\boldsymbol{\ell}}(e_2)=\boldsymbol{\ell}(r_3)+\boldsymbol{\ell}(r_4)=-s$.

Let us start to travel along $D$ starting from a point on $e_1$ by following the orientation on $D$. Using the above argument inductively, we see that the push number of any edge with respect to $\boldsymbol{\ell}$ on our path cannot assume any value other than $s$ or $-s$. Since $D$ is a closed curve this means every edge of $D$ has a push number which is either $s$ or $-s$.
\end{proof}

\begin{lem}
\label{mainlemma}
Let $D$ be a knot diagram, $v$ be an irreducible vertex of $D$, and $\boldsymbol{\ell}$ be a null pattern of a game matrix $M$ of $D$ over $\mathbb{Z}_k$ where  $k\leq \infty$. Then, two non-adjacent regions incident to $v$ are pushed by the same number of times in $\boldsymbol{\ell}$.
\end{lem}

\begin{proof}

First assume  that $D$ is a reduced knot diagram. Let $e_1,...,e_4$ and $r_1,...,r_4$ be the edges and regions incident to $v$, respectively, which are oriented as in Figure \ref{fig:edges}. Without loss of generality we can assume that $\sigma_{\boldsymbol{\ell}}(e_1)=\sigma_{\boldsymbol{\ell}}(e_2)=s$, and $\sigma_{\boldsymbol{\ell}}(e_3)=\sigma_{\boldsymbol{\ell}}(e_4)=-s$ for some $s\in\mathbb{Z}_k$ by Lemma \ref{lempush}. Then, $\boldsymbol{\ell}(r_1)+\boldsymbol{\ell}(r_4)=\sigma_{\boldsymbol{\ell}}(e_1)=\sigma_{\boldsymbol{\ell}}(e_2)=\boldsymbol{\ell}(r_1)+\boldsymbol{\ell}(r_2)$. Hence, $\boldsymbol{\ell}(r_4)=\boldsymbol{\ell}(r_2)$.

\begin{figure}[H]
\centering
\includegraphics[scale=.25]{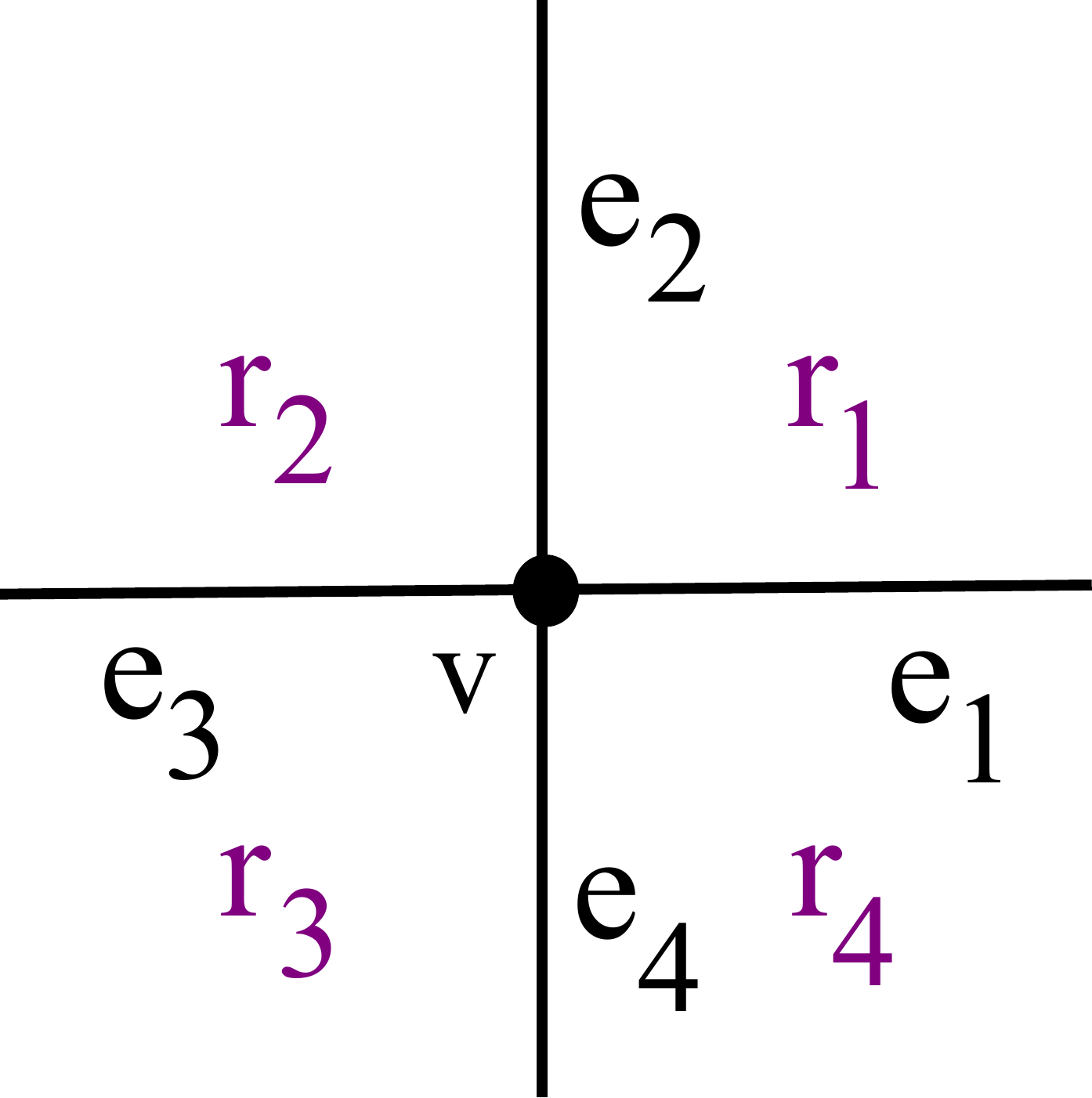}
\caption{Edges and regions that are incident to a vertex}
\label{fig:edges}
\end{figure}

Let now $D$ be any knot diagram which contains reducible crossings. We first endow it with an orientation and construct the link diagram $D'$ obtained from $D$ by applying oriented smoothing operation simultaneously to every reducible vertex of $D$. We illustrate an example of this procedure in Figure \ref{fig:reducible}. Note that the oriented smoothing operation when applied to a reducible vertex preserves the vertex-region  structure of irreducible crossings of the diagram. This means that a game matrix $M'$ of $D'$  can be constructed from $M$ by deleting the rows corresponding to the reducible vertices. Therefore, regions of $D$ and $D'$ can be identified and any null pattern of $M$ is also a null pattern of $M'$, in particular $\boldsymbol{\ell}$. Moreover $D'$ is the union of disjoint components. Let $D''$ be the component of $D'$ which contains $v$. We can construct a game matrix $M''$ of $D''$ by deleting the columns of $M'$ corresponding to the regions whose boundary does not intersect $D''$. Then the restriction $\boldsymbol{\ell}_{res}$ of $\boldsymbol{\ell}$ to the regions of $D''$ is a null pattern of $M''$. Since $D''$ is a reduced knot diagram, by the first part of the proof, two non-adjacent regions incident to $v$ are pushed by the same number of times in $\boldsymbol{\ell}_{res}$, hence in $\boldsymbol{\ell}$.

\begin{figure}[H]
\centering
\includegraphics[scale=.25]{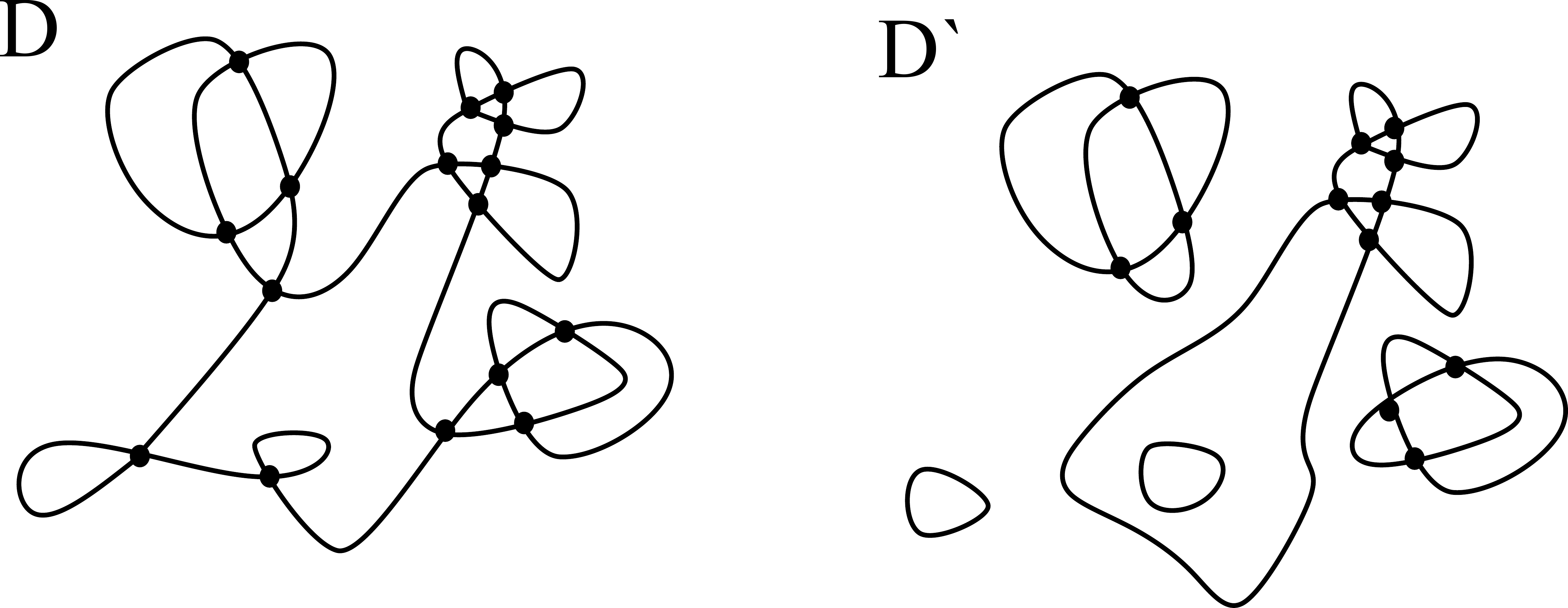}
\caption{A knot diagram containing reducible crossings}
\label{fig:reducible}
\end{figure}

\end{proof}

\begin{prop}
\label{prop0}
Let $D$ be a knot diagram, $M$ be a game matrix of $D$ over $\mathbb{Z}_k$, where $k\leq \infty$. Then, the only null pattern of $M$ where two adjacent regions of $D$ are not pushed is the trivial pattern $\mathbf{0}$.
\end{prop}

\begin{proof}

Let $\boldsymbol{\boldsymbol{\ell}}$ be a null pattern where two adjacent regions $r_1$ and $r_2$ are not pushed. Let $v$ be a vertex incident to both $r_1$ and $r_2$. First assume that $v$ is an irreducible vertex. Let $r_3$, $r_4$ be the other two regions incident to $v$. Since $r_1$ and $r_2$ are not pushed in $\boldsymbol{\ell}$, one of the regions $r_3$ or $r_4$ should not be pushed either by Lemma \ref{mainlemma}. Assume without loss of generality that $r_3$ is not pushed. On the other hand $v$ must be balanced with respect to $\boldsymbol{\ell}$ by Lemma \ref{lem:bal}. Since $r_1,r_2,r_3$ are not pushed, this implies $r_4$ is not pushed either.

Now assume that $v$ is a reducible vertex. Then, there is only one more region, call it $r$, which is incident to $v$ other than $r_1$ and $r_2$. Note that the regions which touch $v$ from one side cannot be adjacent to each other, so either $r_1$ or $r_2$ is the region which touches $v$ from both sides. Hence, $r$ touches $v$ from one side. Therefore the increment number $a$ of $v$ with respect to $r$ is not a zero divisor. Since $r_1$ and $r_2$ are not pushed $(M\boldsymbol{\ell})(v)=a\boldsymbol{\ell}(r)$. On the other hand, since $\boldsymbol{\ell}$ is a null pattern $M\boldsymbol{\ell}=0$. Hence $a\boldsymbol{\ell}(r)=0$. Since $a$ is not a zero divisor we conclude that  $\boldsymbol{\ell}(r)=0$, i.e, $r$ is not pushed.

Using induction on the number of vertices, this argument shows us, by traveling the underlying curve of $D$, starting from the edge incident to $r_1$ and $r_2$ we can never reach a pushed region. Since $D$ is a closed curve, this means that there is no pushed region in $D$, hence $\boldsymbol{\boldsymbol{\ell}}$ is the trivial null pattern  $\mathbf{0}$.
\end{proof}

Now we are ready to state our main result.

\begin{thm}
\label{propadj}
Every knot diagram is always solvable in any version of the $k$-color region select game for all $k\leq\infty$. Moreover, any initial color configuration can be solved uniquely without pushing any two adjacent regions.
\end{thm}

\begin{proof}
Since the difference between the number of regions and number of vertices of a knot diagram is $2$, in the case $k<\infty$, the result follows by Proposition \ref{propmn} and Proposition \ref{prop0}.

In the case $k=\infty$, Let $D$ be a knot diagram with $n$ vertices, $\{r_1,...,r_{n+2}\}$ be an enumeration of the regions of $D$  so that $r_{n+1}$ and $r_{n+2}$ are adjacent. Then take a game  matrix $M$ of $D$ over $\mathbb{Z}$. Let $\widetilde{M}$ be the $n\times n$ matrix obtained from $M$ by deleting its last two columns. Then, Proposition \ref{prop0} implies that $ Ker_\infty(\widetilde{M})=\{\mathbf{0}\}$ (See the proof of Proposition \ref{propmn} for a more detailed explanation). This is equivalent to say that the column vectors of $\widetilde{M}$ (equivalently first $n$ column vectors of $M$) are linearly independent in the $\mathbb{Z}$-module $\mathbb{Z}^n$. Let us denote these column vectors by $\mathbf{c}_1,...,\mathbf{c}_n$. Let $\mathbf{c}$ be an arbitrary vector in $\mathbb{Z}^n$ corresponding to an initial color configuration. It is an elementary fact that $\mathbb{Z}^n$ has rank $n$. Therefore, any set of vectors which has more than $n$ elements is linearly dependent in $\mathbb{Z}^n$. Hence, there are integers $q_1,...,q_n$, and $q$, some of which are nonzero, such that
\begin{equation}
\label{eqnlin}
q_1\mathbf{c}_1+...+q_n\mathbf{c}_n + q \mathbf{c}=\mathbf{0}.
\end{equation}

Note that $q$ cannot be zero, otherwise $\mathbf{c}_1,...,\mathbf{c}_n$ would be linearly dependent. Equation \eqref{eqnlin} is equivalent to the following matrix equation
\begin{equation}
\label{eqnmat}
 M  \begin{bmatrix}
           q_{1} \\
           \vdots \\
           q_{n}\\
           0\\
           0
         \end{bmatrix}= - q \mathbf{c}.
\end{equation}

Multiplying \eqref{eqnlin} by  $-1$ if necessary, we can assume that $q > 0$. Our aim is to show that $q_i$ is divisible by $q$ for $i=1,...,n$. Since this is trivially true if $q=1$, assume further that $q$ is greater than $1$. Then, we can consider the above equation in modulo $q$ and obtain

\begin{equation}
\label{eqnmod}
 \overline{M}  \begin{bmatrix}
           \overline{q_{1}} \\
           \vdots \\
           \overline{q_{n}}\\
           0\\
           0
         \end{bmatrix}= \mathbf{0},
\end{equation}
where $\overline{q_{i}}= q_i$ mod $q$ for $i=1,...,n$ and   $\overline{M}$ is the matrix whose entries are given by $(\overline{M})_{ij}= (M)_{ij}$ mod $q$.

It is easy to observe that $\overline{M}$ is a game matrix of $D$ over $\mathbb{Z}_q$. This observation, together with Proposition \ref{prop0}, immediately implies that $\overline{q_{i}}=0$ for  $i=1,...,n$. So all $q_i$'s are divisible by $q$. Then, there exist numbers  $p_1,..., p_n$ such that $q_i=q p_i$ for $i=1,...,n$, and by equation \eqref{eqnmat} we obtain

\begin{equation}
\label{eqnmat2}
 M  \begin{bmatrix}
           p_{1} \\
           \vdots \\
           p_{n}\\
           0\\
           0
         \end{bmatrix}= - \mathbf{c}.
\end{equation}
Since $M$ is an arbitrary game matrix over $\mathbb{Z}$, and $\mathbf{c}$ is an arbitrary initial color configuration, the above equation means that $D$ is always solvable in any version of the $\infty$-color region select game and any initial color configuration can be solved without pushing any two adjacent regions. Uniqueness follows from the fact that $ Ker_\infty(\widetilde{M})=\{\mathbf{0}\}$.

\end{proof}


\begin{thm}
\label{thmker}
Let $D$ be a knot diagram, $M$ be a game matrix of $D$ over $\mathbb{Z}_k$ where $k< \infty$. Then, $|Ker_k(M)|=k^2$ .
\end{thm}

\begin{proof}
This follows from  Proposition \ref{propker} and Theorem \ref{propadj}.

\end{proof}

\begin{prop}
\label{thmker}
For any knot diagram $D$, there are $k^2$ number of solving push patterns for each initial color configuration in any version of the $k$-color region select game for  $k<\infty$.
\end{prop}

\begin{proof}
This follows directly from Fact \ref{fact3} and Theorem \ref{propadj}.

\end{proof}

We also have the following proposition.

\begin{prop}
\label{propab}
Let $D$ be a knot diagram on which we play a version of the $k$-color region select game, where $k\leq\infty$. Let $a, b \in \mathbb{Z}_k$. Fix two regions adjacent to each other. Then, for any initial color configuration, there is a unique solving pattern where one of the regions is pushed $a$ times and the other is pushed $b$ times. In particular, any null pattern of any game matrix of $D$ over  $\mathbb{Z}_k$ is uniquely determined by its value on two adjacent regions.
\end{prop}

\begin{proof}
Let $M$ be a game matrix of $D$ over $\mathbb{Z}_k$, $\mathbf{c}$ be an initial color configuration of vertices of $D$. Assume that the adjacent regions we fix corresponds to the last two columns of $M$. Then, consider the color configuration
\begin{equation}
\mathbf{\widetilde{c}}:= \mathbf{c}+ M\begin{bmatrix}
           0 \\
           \vdots \\
           0\\
           a\\
           b
         \end{bmatrix}.
\end{equation}
By Theorem \ref{propadj}, there is a unique solving push pattern $(p_1,...,p_n,0,0)^t$ for $\mathbf{\widetilde{c}}$, where $n$ is the number of vertices of $D$. Hence,

\begin{equation}
M\begin{bmatrix}
           p_1 \\
           \vdots \\
           p_n\\
           0\\
           0
         \end{bmatrix}= -\mathbf{c}- M\begin{bmatrix}
           0 \\
           \vdots \\
           0\\
           a\\
           b
         \end{bmatrix},
\end{equation}
which implies $M\mathbf{p}=-\mathbf{c}$, where $\mathbf{p}= (p_1,...,p_n,a,b)^t$. Hence $\mathbf{p}$ is a desired solving pattern for $\mathbf{c}$. For uniqueness, assume that there is another solving pattern $\mathbf{q}:=(q_1,...,q_n,a,b)^t$ for $\mathbf{c}$. Then, $\mathbf{p}-\mathbf{q}$ would be a null pattern of $M$ where two adjacent regions are not pushed. By Proposition \ref{prop0}, $\mathbf{p}=\mathbf{q}$.

\end{proof}

\section{Game on reduced knot diagrams}\label{sec:reduced}

In this section, we examine the $k$-color region select game further for reduced knot diagrams.

\begin{definition}\normalfont

A shading of the regions of a link diagram $D$ is called a \textit{checkerboard shading} if for any pair of adjacent regions of $D$, one of the regions is shaded and the other one is unshaded. It is well-known that all link diagrams admit a checkerboard shading \cite{Ka}.

\end{definition}

\begin{thm}
\label{thm2}
Let $D$ be a reduced knot diagram with $n$ vertices on which we play the $2$-color region select game. Fix a checkerboard shading on $D$. Then, any initial color configuration can be solved uniquely without pushing one shaded and one unshaded region.

In general, there are $2^{n+2-i}$ number of initial color configurations  which can be solved without pushing $i$ number of regions which contains one shaded and one unshaded region. Moreover, there are $2^{n+1-i}$ number of initial color configurations  which can be solved without pushing $i$ number of shaded regions or $i$ number of unshaded regions.

\end{thm}
\begin{proof}
 Take a checkerboard shading of $D$. Consider the following push patterns $\boldsymbol{\ell}_0$, $\boldsymbol{\ell}_1$, $\boldsymbol{\ell}_2$, and $\boldsymbol{\ell}_3$, where $\boldsymbol{\ell}_0$ is the zero pattern; $\boldsymbol{\ell}_1$ is the pattern where only shaded regions are pushed; $\boldsymbol{\ell}_2$ is the pattern where only unshaded regions are pushed; and $\boldsymbol{\ell}_3$ is the pattern where all regions are pushed. It is easy to see that all of these are null patterns of the incidence matrix $M_0(D)$ which corresponds to the $2$-color region select game matrix of $D$. Moreover, they form the set of all nonzero null patterns since $Ker_2(M_0)=4$ by Theorem \ref{thmker}. Note that the only null pattern where at least one shaded and one unshaded region are not pushed is the zero pattern $\boldsymbol{\ell}_0$. The null patterns where any number of unshaded regions are not pushed are $\boldsymbol{\ell}_0$ and $\boldsymbol{\ell}_1$. And lastly, the null patterns where  any number of shaded regions are not pushed are $\boldsymbol{\ell}_0$ and $\boldsymbol{\ell}_2$. Hence, the result follows by Proposition \ref{propmn} .

\end{proof}

\begin{definition}\normalfont
The \textit{distance} $d(r_1,r_2)$ between two regions $r_1$ and $r_2$ of a link diagram $D$ is defined to be the distance between the vertices corresponding to $r_1$ and $r_2$ in the dual graph of $D$.
\end{definition}

\begin{lem}
\label{lemdis}
Let $D$ be a reduced knot diagram and $\boldsymbol{\ell}$ be a null pattern of a game matrix $M$ of $D$ over $\mathbb{Z}_k$ where $k\leq \infty$. Let $s\in \mathbb{Z}_k$ be the push number of some edge $e$ of $D$ with respect to $\boldsymbol{\ell}$. Fix a checkerboard shading on $D$. Let $r_1$ and $r_2$ be two shaded or two unshaded regions. Then $\boldsymbol{\ell}(r_1)= \boldsymbol{\ell}(r_2) +2is$ mod $k$, where $i$ is an integer satisfying $|2i|\leq d(r_1,r_2)$.
\end{lem}

\begin{proof}
Consider the case where $d(r_1,r_2)=2$. So there is a region, call it $r$, which is adjacent to both $r_1$ and $r_2$. Let $e_1$ and $e_2$ be the edges incident to $r_1$, $r$ and $r_2$, $r$, respectively. Then, $\boldsymbol{\ell}(r_1)- \boldsymbol{\ell}(r_2)=\boldsymbol{\ell}(r_1)+\boldsymbol{\ell}(r)-  \boldsymbol{\ell}(r)- \boldsymbol{\ell}(r_2)=\sigma_{\ell}(e_1)-\sigma_{\ell}(e_2)$. On the other hand, $\sigma_{\ell}(e_1)= s$ or $-s$, similarly $\sigma_{\ell}(e_2)= s$ or $-s$ by Lemma \ref{lempush}. Considering every possible case, we obtain $\boldsymbol{\ell}(r_1)- \boldsymbol{\ell}(r_2)=0$, $-2s$, or $2s$.
 The general case follows by applying induction on the distance of $r_1$ and $r_2$.
 \end{proof}

 \begin{thm}
 \label{thmp}
 Let $D$ be a reduced knot diagram on which we play a version of the $k$-color region select game, where $k< \infty$. Fix a checkerboard shading on $D$. Then, for $k=2^n$, $n\in \mathbb{N}$, any initial color configuration can be solved uniquely without pushing one shaded and one unshaded region.

 For other values of $k$, let $p$ be the smallest odd prime factor of $k$.  Then, any initial color configuration can be solved uniquely without pushing one shaded and one unshaded region if the distance between the regions is less than $p$.
 \end{thm}

 \begin{proof}

Let $r_1$ be a shaded and $r_2$ be an unshaded region. Let $M$ be the game matrix of $D$ over $\mathbb{Z}_k$ corresponding to the version of the game we play on $D$. Let $\boldsymbol{\ell}$ be a null pattern of $M$, on which $r_1$ and $r_2$ are not pushed. Let $r$ be a shaded region, adjacent to $r_2$, such that $d(r_1,r_2)=d(r_1,r)+1$. Note that, if $e$ is an edge between $r_2$ and $r$, then $\sigma_{\boldsymbol{\ell}}(e)=\boldsymbol{\ell}(r)$ since $\boldsymbol{\ell}(r_2)=0$. Hence, by Lemma \ref{lemdis}, we have
\begin{equation}
\label{eqn2i}
0=\boldsymbol{\ell}(r_1)= (2i+1)\boldsymbol{\ell}(r) \mod k,
\end{equation}
where $|2i|\leq d(r_1,r)$. If $k=2^n$ for some $n\in \mathbb{N}$, then $2i+1$ mod $k$ cannot be a zero divisor of $\mathbb{Z}_k$, hence  (\ref{eqn2i}) implies $\boldsymbol{\ell}(r)=0$.

For other values of $k$, assume further that $d(r_1,r_2) < p$.  Note that $|2i+1|\leq |2i|+1 \leq d(r_1,r_2)< p $. Hence, $2i+1$ mod $k$  cannot be a zero divisor of $\mathbb{Z}_k$, and therefore  $\boldsymbol{\ell}(r)=0$ for this case as well.

Since $r_2$ and $r$ are adjacent, and  $\boldsymbol{\ell}(r)=\boldsymbol{\ell}(r_2)=0$, we have $\boldsymbol{\ell}=\boldsymbol{0}$ by Proposition \ref{prop0}. Then the result follows by Proposition \ref{propmn}.

 \end{proof}

\subsection{Game on reduced alternating sign diagrams}\label{sec:reducedalternating}

Take a checkerboard shading of a link diagram $L$. Assume that one of the subsets of regions, shaded or unshaded ones, admits an alternating $``+, -"$ signing where every vertex is incident to two regions with opposite signs, as exemplified in Figure \ref{fig:alternating}. Then, the subset of regions which admits such signing is called an \textit{alternating subset of regions}.
\begin{definition}\normalfont
 A link diagram that has an alternating subset of its regions is called an \textit{alternating sign diagram}.
 \end{definition}

 We have the following proposition.

\begin{figure}[H]
\centering
\includegraphics[scale=.25]{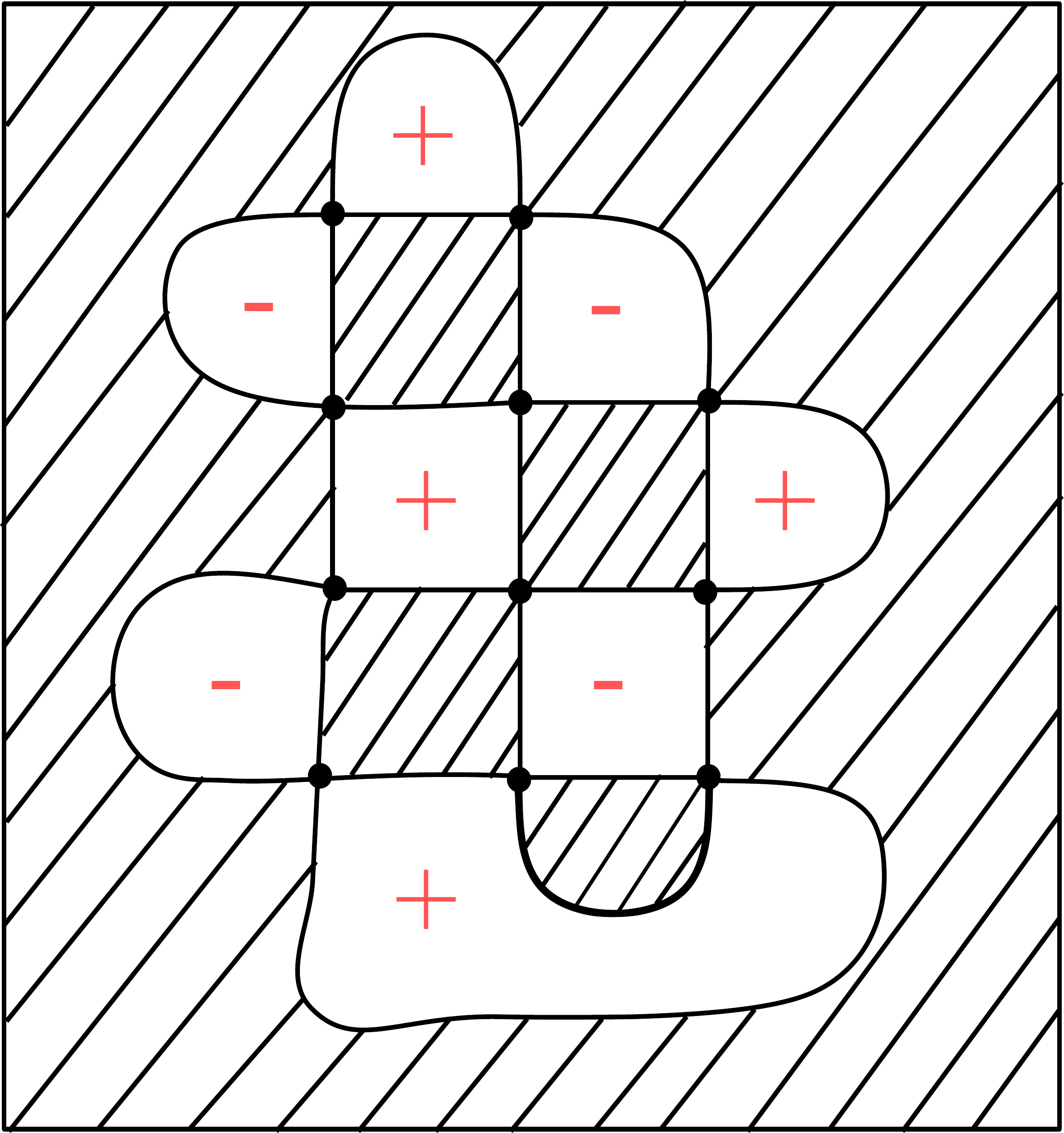}
\caption{An alternating sign diagram}
\label{fig:alternating}

\end{figure}


\begin{prop}
Take a checkerboard shading of a  link diagram $L$. Then, the unshaded regions are alternating if and only if each connected component of the boundary of each shaded
region, except the simple loop ones, have all even number of edges, and vice versa.
\end{prop}
\begin{proof}
$(\Rightarrow)$ Let $\Gamma$ be a connected component of the boundary of a shaded region other than a loop. Take an alternating signing of unshaded regions and sign each edge of $\Gamma$ by the sign of its incident unshaded region. Then the signs of successive edges must be different while we travel along  $\Gamma$ in one direction. Otherwise, the vertex between two successive edges would be incident to two unshaded regions with the same sign,  which contradicts with the definition of the alternating signing. Hence, the signs of edges alternate while we travel along $\Gamma$ in one direction. Since $\Gamma$ is connected this is only possible if $\Gamma$ has even number of edges.

$(\Leftarrow)$ Note that the claim holds true for the link diagrams with zero and one vertex. Suppose the claim holds true for all links with $n-1$ vertices. Now let $L$ be a link with $n$ vertices which satisfies the assumption of the claim. If $L$ does not have any irreducible vertex then it has a vertex on a curl. Removing this vertex with an oriented smoothing as in Figure \ref{fig:orientedsmooth} gives us a link $L'$ with $n-1$ vertices which also satisfies the assumption of the claim. By the induction hypothesis unshaded regions of $L'$ admits an alternating signing. Changing the sign of the region $r$, shown in Figure \ref{fig:orientedsmooth}, if necessary, we see that an alternating signing of unshaded regions of $L'$ induces an alternating signing of unshaded regions of $L$ by reversing the oriented smoothing operation while keeping the sings of the regions. If $L$ has an irreducible vertex $u$, apply a smoothing to $u$ so that the shaded regions incident to $u$ are connected, as shown in Figure \ref{fig:smoothing}. Then the resulting link $L''$ has $n-1$ vertices and it also satisfies the assumption of the claim. By induction hypothesis the unshaded regions of $L''$ admit an alternating signing. Note that the regions $r_1$ and $r_2$, shown in Figure \ref{fig:smoothing} must have opposite signs. Therefore by reversing the smoothing operation while keeping the signs of the unshaded regions of $L''$,  we obtain an alternating signing of the unshaded regions of $L$.
\end{proof}

\begin{figure}[H]
\centering
\includegraphics[scale=.15]{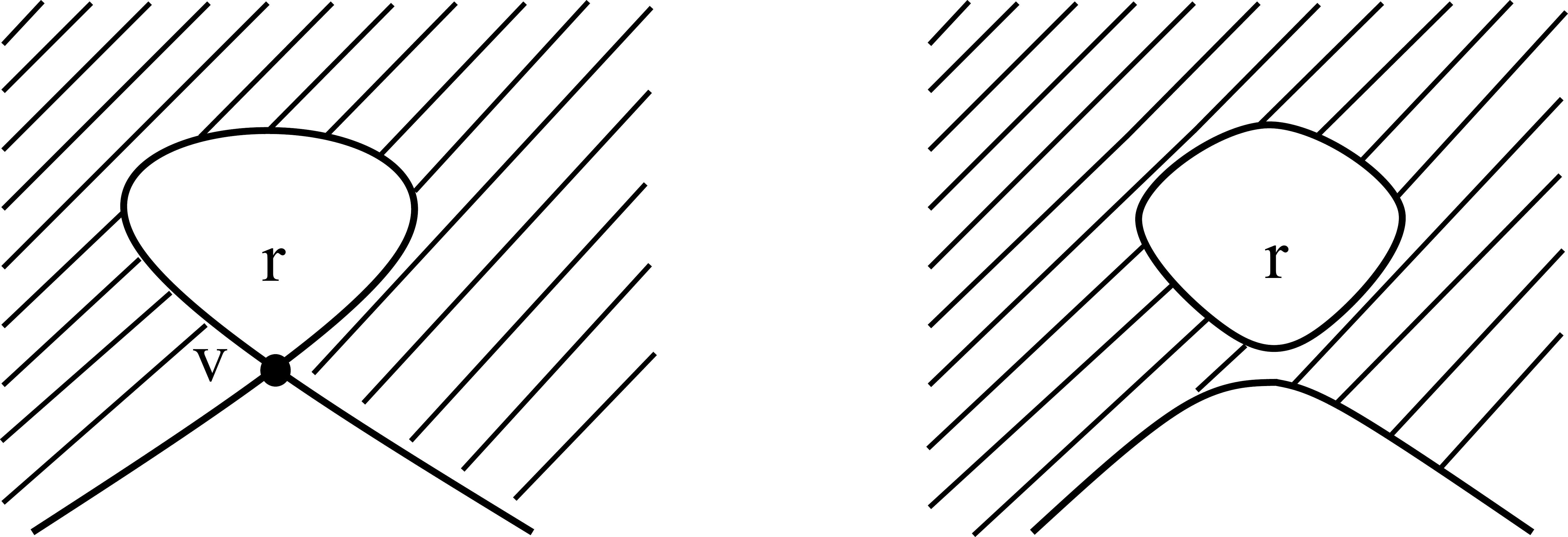}
\caption{Oriented smoothing of a vertex on a curl}
\label{fig:orientedsmooth}

\end{figure}

\begin{figure}[H]
\centering
\includegraphics[scale=.3]{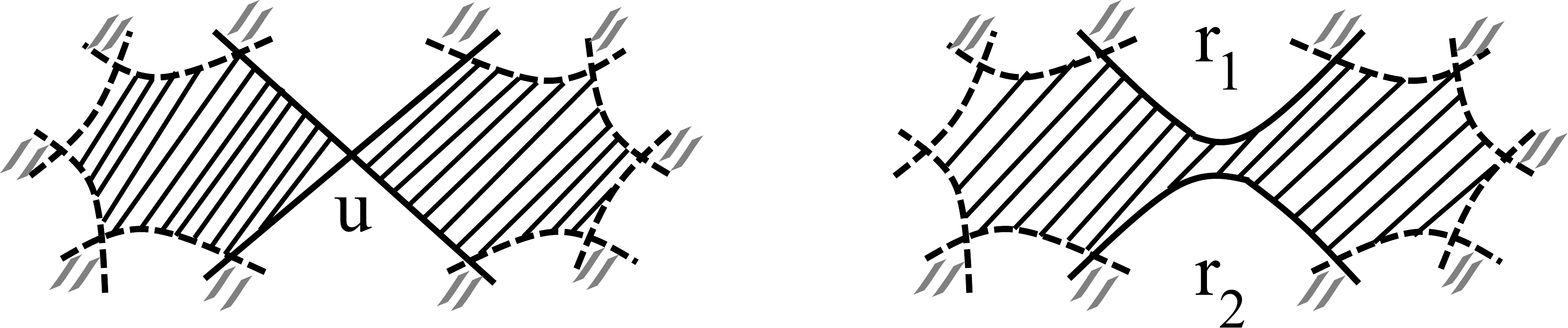}
\caption{A smoothing of an irreducible vertex}
\label{fig:smoothing}

\end{figure}

\begin{cor}
Take a checkerboard shading of a knot diagram $D$. Then, the unshaded regions are alternating if and only if all shaded regions are incident to even number of edges, and vice versa.
\end{cor}

\begin{thm}
\label{thmk}
Let $D$ be a reduced knot diagram with $n$ vertices on which we play a version of the $k$-color region select game, where $k\leq \infty$. Assume that $D$ admits an alternating signing of its unshaded regions in a checkerboard shading of $D$. Then the followings hold.

1) Any initial color configuration can be solved uniquely without pushing one shaded and one unshaded region.

2) If $k$ is an odd number, then any initial color configuration can be solved uniquely without pushing two unshaded regions with opposite signs.

3) In general, let $S$ be a set of $i$ number of regions, and $q$ be the number of initial color configurations which can be solved  without pushing the regions in $S$. Then, $q=k^{n+2-i}$ if $S$ contains one shaded and one unshaded region. $q=k^{n+1-i}$ if $S$ consists of shaded regions or unshaded regions with the same sign. $q=k^{n+2-i}$ for $k$ is odd, and $q=k^{n+2-i}/2$ for $k$ is even if $S$ consists of unshaded regions not all of which have the same sign.

\end{thm}

\begin{proof}

 We start by proving claim (3). Fix a version of the $k$-color region select game and let $M$ be the corresponding game matrix of $D$ over $\mathbb{Z}_k$.  By Proposition \ref{propmn},  $q=k^{n+2-i}/j$ where $j$ is the number of null patterns of $M$ where the regions in $S$ are not pushed. Hence we just need to determine what $j$ is for each case. To do that let us investigate the form of null patterns. Let $a,b\in \mathbb{Z}_k$, and consider the push pattern $\boldsymbol{\ell}_{a,b}$ where $\boldsymbol{\ell}_{a,b}(r)=a$ if $r$ is a shaded region, $\boldsymbol{\ell}_{a,b}(r)=b$ if $r$ is an unshaded region with a plus sign, and $\boldsymbol{\ell}_{a,b}(r)=-b-2a$ (mod $k$) if $r$ is an unshaded region with a minus sign. Then $\boldsymbol{\ell}_{a,b}$ is a null pattern of $M$ for every choice of $a$ and $b$ because $D$ is a reduced knot diagram. Moreover, by Proposition \ref{propab}, all null patterns of $M$ must be in the form of $\boldsymbol{\ell}_{a,b}$.

 If one shaded and one unshaded region are not pushed in a null pattern $\boldsymbol{\ell}_{a,b}$, then $a$ and $b$ must be zero which corresponds to the zero null pattern. Hence $j=1$ and $q=q=k^{n+2-i}$. 

If a set of shaded regions are not pushed in a null pattern $\boldsymbol{\ell}_{a,b}$, then $a=0$ and $b$ can take any value in $\mathbb{Z}_k$. So there are $k$ number of such null patterns. Hence, $j=k$ and  $q=k^{n+1-i}$.

If a set of unshaded regions with the same signs are not pushed in a null pattern $\boldsymbol{\ell}_{a,b}$, then either $b=0$ or $b+2a=0$. In both cases, value of $a$ determines all possible null patterns. So there are $k$ number of such null patterns. Hence, $j=k$ and  $q=k^{n+1-i}$.

 If a set of unshaded regions are not pushed in a null pattern $\boldsymbol{\ell}_{a,b}$ and at least two regions have opposite signs, then $b=0$ and $2a=0$. If $k$ is odd, then $2$ cannot be a zero divisor of $\mathbb{Z}_k$, hence $a=0$ as well. Then $\boldsymbol{\ell}_{a,b}$ corresponds to the zero null pattern, i.e., $j=1$.  On the other hand, if $k$ is even, then $2a=0$ has two solutions $a=0$ and $a=k/2$, which correspond to two null patterns, i.e., $j=2$. Hence, $q=k^{n+2-i}$ if $k$ is odd, and  $q=k^{n+2-i}/2$ if $k$ is even. This completes the proof of claim (3).

Claim (3) implies claim (2) and claim (1) in the case $k<\infty$. Hence it remains to prove claim (1) in the case $k=\infty$. Note that the patterns $\boldsymbol{\ell}_{a,b}$ when $a,b\in \mathbb{Z}$ form the set of all null patterns of any game matrix of $D$ over $\mathbb{Z}$, as well. Hence, for all $k\leq \infty$ the only null pattern of any game matrix of $D$ over $\mathbb{Z}_k$ where one shaded and one unshaded region are not pushed is the trivial pattern $\mathbf{0}$. In other words, Proposition \ref{prop0} still holds true when we replace two adjacent region by one shaded and one unshaded region in the case $D$ is an alternating sign diagram. Therefore, we can repeat the proof of Theorem \ref{propadj} replacing two adjacent regions by one shaded and one unshaded region. Hence, the result follows.




\end{proof}

\section{Game on knot diagrams}\label{sec:assertions}

Recall that in the proof of Lemma \ref{mainlemma}, for a knot diagram $D$ we constructed the link diagram $D'$ by applying the oriented smoothing operation simultaneously to every reducible vertex of $D$. Let us call this operation the \emph{reducing operation} on a knot diagram.

\begin{definition}\normalfont
We call a sub-diagram $P$ of $D$ a \emph{connectedly reducible part} of $D$ if it satisfies the following two conditions:

1) It stays connected after applying the reducing operation to $D$.

2) No sub-diagram containing it stays connected after the reducing operation.
\end{definition}

Equivalently, $P$ is a connectedly reducible part of $D$ if and only if it corresponds to a reduced (disjoint) component of $D'$ after the reducing operation. See Figure \ref{fig:six} for an illustration.

\begin{figure}[H]
\centering
\includegraphics[scale=.2]{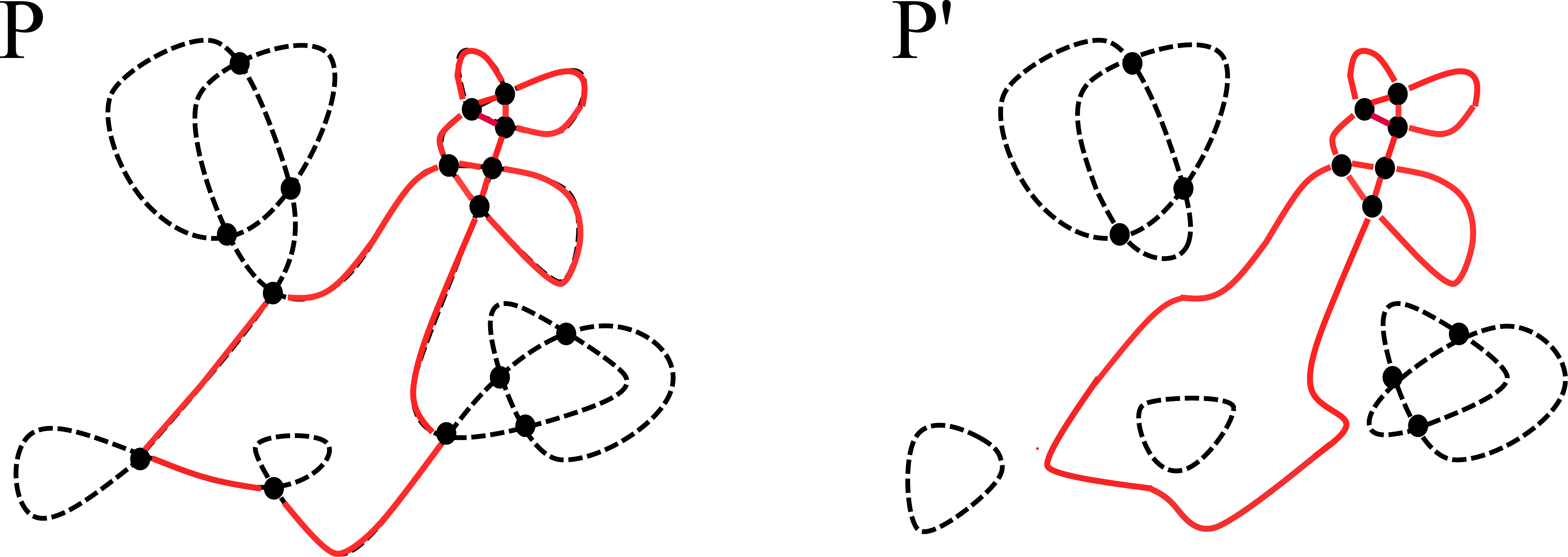}
\caption{A connectedly reducible part of a knot diagram (in red)}
\label{fig:six}

\end{figure}

Now let $P$ be a connectedly reducible part of $D$ and denote the reduced component of $D'$ corresponding to $P$ by $P'$. Then $P$ and $P'$ have the same irreducible vertex-region incidence structure and we can identify their regions as we identify the regions of $D$ and $D'$. Let $M$ be a game matrix of $D$ over $\mathbb{Z}_k$, $k\leq \infty$. Delete all rows of $M$ corresponding to reducible vertices and columns corresponding to the regions whose boundary does not intersect $P$. The resulting matrix, call it $M'$ is a game matrix of $P'$. Note that if $\boldsymbol{\ell}$ is a null pattern of  $M$, then its restriction $\boldsymbol{\ell}_{res}$ to the regions of $P$  is a null pattern of $M'$. Moreover, two different null patterns $\boldsymbol{\ell}$ and $\boldsymbol{\kappa}$ of $M$ cannot have the same restriction. Indeed, if $\boldsymbol{\ell}_{res}=\boldsymbol{\kappa}_{res}$, then $\boldsymbol{\ell}-\boldsymbol{\kappa}$ would be a null pattern of $M$ which is zero on the regions of $P'$, equivalently on the regions of $P$. But $P$ contains at least two adjacent regions of $D$. $\boldsymbol{\ell}=\boldsymbol{\kappa}$ by Proposition \ref{prop0}.

So the restriction map of the null patterns of $M$ to the null patterns of $M'$ is injective. Moreover, it is also surjective. Indeed, let $\boldsymbol{\alpha}$ be a null pattern of $M'$, $r_1$ and $r_2$ be two adjacent regions of $P'$ (equivalently $P$ and $D$). Then, by Proposition \ref{propab}, there is a null pattern $\boldsymbol{\beta}$ of $M$ which agrees with $\boldsymbol{\alpha}$ on $r_1$ and $r_2$. However, then both $\boldsymbol{\alpha}$ and $\boldsymbol{\beta}_{res}$ are null patterns of $M'$ which agree on two adjacent regions. Again by Proposition \ref{propab}, $\boldsymbol{\alpha}=\boldsymbol{\beta}_{res}$. From this one to one correspondence we obtain the following observation.

\emph{Observation}: Let $S$ be a subset of regions of $P'$. Then, the number of null patterns of $M'$ where the regions of $S$ are not pushed is equal to the number of null patterns of $M$ where the regions of $S$ are not pushed.

Theorem \ref{thm2}, \ref{thmp}, and \ref{thmk} determine how many initial color configurations there are which can be solved without pushing certain regions in any version of the $k$-color region select game when played on reduced knot diagrams with some conditions on $k$ or on the diagram. When we look at the proofs we see that this number is determined by the number of null patterns of the corresponding game matrix where these regions are not pushed. Therefore, by the above observation, and together with the fact that $P'$ is a reduced diagram whose regions can be identified by the regions of $P$, the proofs of Theorem \ref{thm2}, \ref{thmp} and \ref{thmk} imply the following result for general knot diagrams.

 \begin{thm}
 Let $D$ be a knot diagram on which we play a version of the $k$-color region select game. Fix a checkerboard shading on $D$. Let $P$ be a connectedly reducible part of $D$. Then all the assertions of Theorem \ref{thm2} for  $k=2$  and Theorem \ref{thmp} for $k<\infty$, hold true if the regions in the assertions belong to $P$.

Assume further that $P$ admits an alternating signing of its unshaded regions. Then all the assertions of Theorem \ref{thmk} for $k\leq\infty$  hold true if the regions in the assertions belong to $P$.
 \end{thm}

\bibliographystyle{plain}

 \end{document}